\documentclass{amsart}

\usepackage{latexsym,amsfonts,amssymb}
\usepackage{mathrsfs}
\usepackage{graphicx,epsfig,subfigure}
\usepackage{psfrag}
\usepackage{enumitem}
\usepackage{colortbl}
\usepackage{hyperref} 

\newtheorem{theorem}{Theorem}[section]

\newtheorem{proposition}[theorem]{Proposition}
\newtheorem{conjecture}[theorem]{Conjecture}
{\theoremstyle{definition}
\newtheorem{definition}[theorem]{Definition}
\newtheorem{remark}[theorem]{Remark}
\newtheorem{example}[theorem]{Example}}

\newcommand{\N}{\mathbb N}

\newcommand{\R}{\mathbb R}
\newcommand{\C}{\mathbb C}

\newcommand{\Jac}{\textnormal{Jac}}

\title[On topological approaches to the Jacobian conjecture in $\C^n$]{On topological approaches to the Jacobian conjecture in $\C^n$}

\author[F. Braun, L.R.G. Dias and J. Venato-Santos]
{Francisco Braun$^{1}$, Luis Renato Gon\c{c}alves Dias$^{2}$ and Jean Venato-Santos$^{3}$}

\address{$^1$ Departamento de Matem\'{a}tica, Universidade Federal de S\~ao Carlos,
13565-905 S\~ao Carlos, S\~ao Paulo, Brazil}
\email{franciscobraun@dm.ufscar.br}

\address{$^2$ and $^3$ Faculdade de Matem\'{a}tica, Universidade Federal de Uberl\^{a}ndia,
38408-100 Uberl\^{a}ndia, Minas Gerais, Brazil}
\email{lrgdias@ufu.br and jvenatos@ufu.br}



\date{\today}

\begin{document}
	
\maketitle

\begin{abstract} We obtain a structure theorem  for the nonproperness set $S_f$ of a nonsingular polynomial mapping $f:\C^n \to \C^n$. 
Jelonek's results on $S_f$ and our result show that if $f$ is a counterexample to the Jacobian conjecture, then $S_f$ is a hypersurface  such that $S_f\cap Z \neq \emptyset$, for any $Z\subset \C^n$ biregular to $\C^{n-1}$ and $Z = h^{-1}(0)$ for a polynomial submersion $h: \C^n \to \C$. 
Also, we present topological approaches to the Jacobian conjecture in $\C^n$. 
In particular, these conditions are used to extend bidimensional results of Rabier and L\^e and Weber to higher dimensions. 
\end{abstract}

\section{Introduction and statement of the main results}

The main results and the proofs of this paper are strongly motivated by the arguments presented in Krasi\'nski and Spodzieja paper \cite{KS}. 

Let $g=(g_1, g_2,\ldots, g_m): \C^n \to \C^m$ be a holomorphic mapping.
We denote by $\Jac(g)(x)$ the Jacobian matrix of $g$ at $x$. 
When $m = 1$, we  denote this matrix by $\nabla g(x)$. 
In case $m = n$, we denote by $\det \Jac(g)(x)$ the determinant of the Jacobian matrix of $g$ at $x$. 
A point $y \in \C^m$ is a \emph{regular value} of $g$ if for each $x \in g^{-1} (y)$ the matrix $\Jac(g)(x)$ has maximum rank. 
We say that $g$ is \emph{nonsingular} if its range contains only regular values. 
Let $J = \left(i_1, i_2, \ldots, i_\ell \right)$, $  i_1 < i_2 < \ldots < i_\ell$, be  a sequence of integers in $\{1, 2, \ldots, m\}$. We denote by $G_J$ the mapping $G_J = \left(g_{i_1}, g_{i_2}, \ldots, g_{i_\ell} \right): \C^n \to \C^\ell$. 
When $J = \left(1, \ldots, k-1, k+1, \ldots m\right)$, we denote $G_{J}$ by $G_{\widehat{k}}$.  

A mapping $g: \C^n \to \C^m$ is said to be \emph{proper at $y\in\C^m$} if there exists a neighborhood $V$ of $y$ such that $g^{-1}\left(\overline{V}\right)$ is compact. 
The set of points at which $g$ is not proper is denoted by $S_g$. 
We say that $g$ is \emph{proper} if $S_g$ is the empty set $\emptyset$. 
The set $S_g$ has been considered in many problems and applications, see for instance \cite{J93, J99, J, JK, KOS}. 

In this paper we deal with nonsingular polynomial mappings $f: \C^n \to \C^n$, which in this case means that $\det \Jac(f)(x)$ is a non-zero constant. 
The claim that $f$ is a polynomial automorphism is the very known \emph{Jacobian conjecture}, which remains unsolved until these days, see for instance \cite{BCW, E} for details. 
From the well known Hadamard's global inversion theorem and the main result of Cynk and Rusek \cite{CR}, $f$ is an automorphism if and only if it is nonsingular and $S_f$ is the empty set. 
So the Jacobian conjecture will be proved if one shows that $S_f$ is the empty set for any nonsingular polynomial mapping $f:\C^n \to \C^n$. 
The following is a new result on the structure of $S_f$, whose proof is given in Section \ref{s:2}. 
Recall that a mapping $\phi: X \to Y$,  $X \subset \C^n$ and $Y\subset \C^m$ algebraic sets, is a \emph{regular mapping} if $\phi$ is the restriction to $X$ of a polynomial mapping defined in $\C^n$. 
We say that $\phi: X \to Y$ is a \emph{biregular mapping} if $\phi$ and $\phi^{-1}$ are regular mappings, in this case we say that \emph{$X$ is biregular to $Y$}.

\begin{theorem}\label{t:hyper} 
	Let $f: \C^n\to \C^n$ be a nonsingular polynomial mapping. 
	Let $Z = h^{-1}(0)$ biregular to $\C^{n-1}$, where $h: \C^n \to \C$ is a nonsingular polynomial, such that 
	$S_f \cap Z = \emptyset$. Then $f$ is an automorphism.
\end{theorem}

Testing sets $Z$ for Theorem \ref{t:hyper} include, for instance, graphs of polynomial functions of $\C^{n-1}$. 
From Jelonek \cite{J93, J99}, it follows that in this case of nonsingular polynomial mappings $f:\C^n \to \C^n$, the set $S_f$ is either empty or a hypersurface. 
Therefore, if $f$ is a counterexample to the Jacobian conjecture, then the hypersurface $S_f$ is such that $S_f\cap Z\neq \emptyset$ for any algebraic set $Z\subset \C^n$ satisfying the assumptions of the above theorem.

\smallskip 

On the other hand we recall that a continuous mapping $g: X \to Y$ between topological spaces $X$ and $Y$ is a \emph{trivial fibration} if there exist a topological space $\mathcal{F}$ and a homeomorphism $\varphi\colon \mathcal{F}\times Y\to X$ such that $pr_2 = g\circ\varphi$ is the second projection on $Y$. 
We say further that $g$ is a \emph{locally trivial fibration at $y\in Y$} if there exists an open neighborhood $U$ of $y$ in $Y$ such that $g|_{g^{-1}(U)}\colon g^{-1}(U)\to U$ is a trivial fibration. 
We denote by $B(g)$ the set of points of $Y$ where $g$ is not a locally trivial fibration. 
The set $B(g)$ is usually called the \emph{bifurcation} (or \emph{atypical}) set of $g$. 
In case $B(g)$ is the empty set we simply say that $g$ is a \emph{locally trivial fibration}. 

In case $n=2$, as a consequence of Abhyankar and Moh embedding theorem \cite{AM}, L\^e and Weber \cite{LW} presented  the following result. 

\begin{theorem}[\cite{LW}]\label{t:LW}
Let $f = (f_1,f_2)\colon \C^2 \to\C^2$ be a nonsingular polynomial mapping. 
If $B\left(f_1\right) = \emptyset$, then $f$ is an automorphism.
\end{theorem}

As a consequence, they obtained the following geometrical-tolological formulation of the Jacobian conjecture in $\C^2$: 

\begin{conjecture}[\cite{LW}]\label{con:LW}
Let $f_1 \colon \C^2 \to\C$ be a polynomial function. 
If $B\left(f_1\right) \neq \emptyset$, then for any polynomial function $f_2\colon\C^2 \to \C$ there exists $x \in \C^2$ such that $\det \Jac(f_1,f_2)(x) = 0$. 
\end{conjecture}

Analytical conditions ensuring locally trivial fibrations are known in the literature. 
So, in view of Theorem \ref{t:LW}, for example, it is expected the use of such conditions to obtain particular cases of the Jacobian conjecture. 
In this context, Rabier \cite{Ra} considered analytical conditions to define the set $\widetilde K_{\infty}(g)$ for holomorphic mappings $ g: \C^n \to \C^m$ (see Definition \ref{d:ra-ity}). 
He then obtained the following result: 
\begin{theorem}[{\cite[Th. 9.1]{Ra}}]\label{t:Ra} Let $f = (f_1,f_2):\C^2 \to \C^2$ be a polynomial mapping.
\begin{enumerate}[label={(\alph*)}]
\rm	\item\it If $f$ is an automorphism, then $f$ is nonsingular and $\widetilde K_{\infty}(f_1)=\emptyset $ and $\widetilde K_{\infty}(f_2)=\emptyset$.
\rm \item \it If $f$ is nonsingular and $\widetilde K_{\infty}(f_1)=\emptyset$, then $f$ is an automorphism.
\end{enumerate}		
\end{theorem} 

Rabier also showed that for holomorphic mappings $g: \C^n \to \C^m$, $B(g) \subset \widetilde K_{\infty} (g)$. 
Clearly if $f: \C^n \to \C^n$ is a polynomial automorphism, then $f$ is nonsingular and $B (F_{\widehat{k}})= \emptyset$ for each $k \in \{1,2, \ldots, n\}$. 
On the other hand, Example \ref{ex:noGPS} below shows that for $n\geq 3$ a version of Theorem \ref{t:Ra}-(a) does not hold for the condition $\widetilde K_{\infty} (F_{\widehat{k}}) = \emptyset$, see also Remark \ref{r:gra-lw}. 
It is known that  a locally trivial fibration $g: \C^n \to \C^m$ has its fibers simply connected (see Proposition \ref{p:global}). % and also Theorem \ref{teoR2}). 
So it turns out that next two results generalize in different manner Theorem \ref{t:LW}  and part (b) of Theorem \ref{t:Ra} to higher dimensions. 

\begin{theorem}\label{t:Ra-topV}
Let $f =(f_1,f_2,...,f_n):\C^n \to \C^n$ be a nonsingular polynomial mapping. 
If the fibers of $F_{\widehat{n}}$ are simply connected out of $B\left(F_{\widehat{n}}\right)$, then $f$ is an automorphism. 
In particular, if $B\left(F_{\widehat{n}}\right) = \emptyset$ then $f$ is an automorphism.
\end{theorem}

\begin{theorem}\label{t:main}
	Let $f = (f_1,f_2,...,f_n):\C^n \to\C^n$ be a nonsingular polynomial mapping.
	Assume that the connected components of the fibers of $F_{\widehat{k}}: \C^n \to \C^{n-1}$ out of $B(F_{\widehat{k}})$ are simply connected, for all $k\in \{2, \ldots, n\}$.
	Then $f$ is an automorphism. 
\end{theorem} 

We point out that it is enough to test the simply connectedness in the hypotheses of above theorems over any open set of $\C^{n-1}$. 
This is so because $B\left(F_{\widehat{k}}\right)$ is always contained in a hypersurface of $\C^{n-1}$, see details in the proof of Theorem \ref{t:proper}. 

In Section \ref{s:2}, we relate to each other the sets $S_f$ and $B\left(F_{\widehat{k}}\right)$ in Theorem \ref{t:proper} and apply it in the proofs of theorems \ref{t:Ra-topV} and \ref{t:main}. 

Now analogously to Conjecture \ref{con:LW}, as a direct application of Theorem \ref{t:Ra-topV} we obtain the following equivalence of the Jacobian conjecture in $\C^n$: 

\begin{conjecture}
Let $F_{\widehat{n}} = (f_1, f_2, \ldots, f_{n-1}) \colon \C^n \to \C^{n-1}$ be a polynomial mapping. 
If $B(F_{\widehat{n}}) \neq \emptyset$, then for any polynomial function $f_n\colon \C^n\to\C$ there exists $x \in \C^n$ such that $\det \Jac\left(F_{\widehat{n}}, f_n \right)(x) = 0$.
\end{conjecture}

Another application of Theorem \ref{t:Ra-topV} is a topological proof of the bijectivity of the nonsingular mappings $I+H:\C^4\to\C^4$, with $I$ the identity and $H$ a homogeneous polynomial of degree three, that appear in Hubbers classification \cite{hubbers}, see Remark \ref{r:Hu} for details. 

We end the paper with a result in $\C^2$. 

\begin{proposition}\label{teoR2}
Let $f_1: \C^2 \to \C$ be a nonsingular polynomial function.
The following statements are equivalent:
\begin{enumerate}[label={(\alph*)}]
\rm\item\label{en:a}\it The connected components of a fiber $f_1^{-1}(c)$ are simply connected. 
\rm\item\label{en:b}\it There exists a polynomial $f_2: \C^2 \to \C$ such that the mapping $(f_1, f_2): \C^2 \to \C^2$ is an automorphism.
\end{enumerate}
\end{proposition} 

Therefore the assumptions on a nonsingular polynomial function $f_1 :\C^2 \to\C$ in theorems \ref{t:Ra-topV}, \ref{t:main} and Proposition \ref{teoR2} are equivalent. 
 Moreover, these assumptions are equivalent to $B\left(f_1\right) = \emptyset$. 
From the results of Krasi\'nski and Spodzieja \cite[Theorem 4.1]{KS}, it also follows that the above conditions are equivalent to the Hamiltonian vector field of $f_1$ defined in $\C[x]$ to be onto $\C[x]$. 

%%%%%%%%%%%%%%%%%%%%%%%%%%%%%%%%%%%%%%%%%%%%%%%%%%%%%%%%

\section{Proofs of the theorems}\label{s:2}

The following is a well known result on fibrations, see for instance \cite[11.6]{St51}. 

\begin{proposition}\label{p:global}
If $g\colon X\to Y$ is a locally trivial fibration and $Y$ is a contractible space, then $g$ is a trivial fibration.
\end{proposition}

The following proposition will be used in the sequel to prove our main results.  

\begin{proposition}\label{p:biregC}
Let $f = (f_1,f_2,...,f_n):\C^n \to\C^n$ be a nonsingular polynomial mapping and $k  \in \{1,2, \ldots, n\}$.
Let $Y$ be a connected component of a fiber of $F_{\widehat{k}} $.
If $Y$ is a simply connected set, then $f_k|_Y: Y \to \C$ is a biregular function.
\end{proposition}
\begin{proof}
In this proof we follow a reasoning of \cite[p. 310]{KS}. 
Since $Y$ is a simply connected nonsingular algebraic curve, it follows from the Riemann mapping theorem that it is biholomorphic to $\C$.
We call $\phi : Y \to \C$ this biholomorphism. 
From \cite[Theorem 4]{RW} it follows that $\phi$ is a biregular mapping.
The composite function $g = f_k  \circ \phi^{-1}: \C \to \C$ is a polynomial locally invertible function, i.e. $g(z) = a z + b$, with $a, b\in \C$ and $a\neq 0$.
Therefore, $f_k|_Y = g \circ \phi$ is a biregular mapping, which completes the proof. 
\end{proof}

Let $f:\C^n \to \C^n$ be a nonsingular polynomial mapping. 
From \cite{J93, J99}, we know that the nonproperness set $S_f$ is either empty or it is a hypersurface. 
Then $\C^n \setminus S_f$ is a connected subset of $\C^n$ (see for instance \cite[Lemma 8.1]{J}). 
Since $f$ is a local homeomorphism, it follows that $f$ is a dominant mapping and it is an analytic cover of geometric degree $\mu(f)$ on $\C^n \setminus S_f$. 
Thus $\# f^{-1}(y) = \mu(f)$ for any $y\in \C^n\setminus S_f$.
Now, the fact that $f$ is a local homeomorphism implies that $\# f^{-1}(z) < \mu(f)$, for any $z\in S_f$. Therefore, we have
\begin{equation}\label{Jf}
S_f = \left\{ y \in \C^n \mid \# f^{-1}(y) \neq \mu(f) \right\}.
\end{equation}
For any $1\leq k \leq n$, we denote by $\pi_{\widehat{k}}: \C^n \to\C^{n-1}$ the projection $\pi_{\widehat{k}}(x_1,...,x_n)=(x_1,..,x_{k-1},x_{k+1},..,x_n)$. 
With the above notations we have:

\begin{theorem}\label{t:proper}
Let $f = (f_1,f_2,...,f_n):\C^n \to\C^n$ be a nonsingular polynomial mapping and $k \in \{1,2, \ldots, n\}$.
Assume that the connected components of the fibers of $F_{\widehat{k}}$ out of $B(F_{\widehat{k}})$ are simply connected. 
Then
\begin{enumerate}[label={(\alph*)}]
\rm\item\label{enSfBf}\it $S_f \subset \pi_{\widehat{k}}^{-1}(Z)$ for any algebraic set $Z\subset \C^{n-1}$ such that $B(F_{\widehat{k}})\subset Z$.
		
\vspace{0.15cm}
		
\rm\item\label{enSf} \it $S_f = \pi_{\widehat{k}}^{-1}\left(\pi_{\widehat{k}}(S_f)\right).$
\end{enumerate}
\end{theorem}
\begin{proof} 
We begin with the proof of \ref{enSfBf}. 
Let $Z\subset \C^{n-1}$ be an algebraic set such that $B(F_{\widehat{k}})\subset Z$. 
If $Z = \C^{n-1}$ there is nothing to prove. 
So assume $Z \neq \C^{n-1}$. 
It follows that $L := \C^{n-1} \setminus Z$ is an open connected set (see \cite[Lemma 8.1]{J}) such that $F_{\widehat{k}}|_{F_{\widehat{k}}^{-1}\left( L\right)} : F_{\widehat{k}}^{-1}\left( L\right) \to L$ is a locally trivial fibration, and hence there exists $d_k \in \N$ such that $F_{\widehat{k}}^{-1}(\widetilde{y})$ has $d_k$ connected components for each $\widetilde{y} \in L$.
	
Now let $y \in \pi_{\widehat{k}}^{-1} (L)$
and  $V_1 \cup \cdots \cup V_{d_k}$ the decomposition of $F_{\widehat{k}}^{-1} \left(\pi_{\widehat{k}}(y)\right)$ into its connected components.  From Proposition \ref{p:biregC}, it follows that $f_k|_{V_j} : V_j \to \C$ is a biregular function for each $j = 1, \ldots, d_k$. 
This shows that 
\begin{equation}\label{e90}
\# f^{-1}(y) = d_k, \ \forall y \in \pi_{\widehat{k}}^{-1} (L). 
\end{equation}
Since $S_f$ is an hypersurface and $ \pi_{\widehat{k}}^{-1} (L)$ is open, it follows that $\pi_{\widehat{k}}^{-1} (L)\setminus S_f \neq \emptyset$, which by \eqref{Jf} and \eqref{e90} gives that $d_k = \mu(f)$. 
Therefore it follows that $\pi_{\widehat{k}}^{-1}(L) \subset \C^n \setminus S_f$, proving \ref{enSfBf}. 

Now we prove statement \ref{enSf}. 
We know that $B (F_{\widehat{k}})$ is contained in an algebraic hypersurface $Z\subset \C^{n-1}$, see for instance the main result of \cite{J03} or \cite[Corollaire 5.1]{Ve}. 
Let $ Z_1 \cup \ldots \cup Z_l$ be the decomposition of $Z$ into its irreducible components. 
It follows that $\pi_{\widehat{k}}^{-1}(Z_1) \cup \ldots \cup \pi_{\widehat{k}}^{-1}(Z_l)$ is the decomposition of $\pi_{\widehat{k}}^{-1}(Z)$ into its irreducible components. 
By statement \ref{enSfBf} $S_f \subset \pi_{\widehat{k}}^{-1}(Z)$, and since $S_f$ and $\pi_{\widehat{k}}^{-1}(Z)$ are hypersurfaces, it follows that there are indices $i_1, \ldots, i_j \in \{1, 2, \ldots, l\}$ such that 
$$
S_f = \pi_{\widehat{k}}^{-1}(Z_{i_1}) \cup \cdots \cup \pi_{\widehat{k}}^{-1}(Z_{i_j}) = \pi_{\widehat{k}}^{-1}\left( Z_{i_1} \cup \cdots \cup Z_{i_j}\right). 
$$ 
Therefore $\pi_{\widehat{k}}^{-1} (\pi_{\widehat{k}}(S_f)) = S_f$, proving statement \ref{enSf}. 
\end{proof}

We can also give the 

\begin{proof}[Proof of Theorem  \ref{t:hyper}] 
We have $Z = h^{-1}(0)$, for an irreducible nonsingular polynomial $h: \C^n \to \C$. 
Thus $V := f^{-1}(Z) = g^{-1}(0)$, where $g := h \circ f$, is a smooth algebraic set.
It follows from the assumption and \eqref{Jf} that the restricted mapping $f|_V : V \to Z$ is a cover mapping with degree $\mu(f)$. 
The connected components of $V$, say $V = V_1 \cup \cdots \cup V_{\mu(f)}$, are the irreducible components of $V$.
Let $q_j$ be an irreducible polynomial such that $V_j = q_j^{-1}\{0\}$, $j = 1, \ldots, \mu(f)$.
Since $g$ is nonsingular, it follows that $g = \gamma q_1 \cdots q_{\mu(f)}$, for a suitable $\gamma \in \C$.
	
Since $Z$ is simply connected, the restrictions $f|_{V_i}$ are biholomorphisms onto $Z$. 
So $f|_{V_i} : V_i \to Z \subset \C^n$ are proper regular mappings, and hence it follows by \cite[Theorem 3-(ii)]{RW} that $f|_{V_i}^{-1}: Z \to V_i$ are regular mappings. 
Therefore $f|_{V_i}$ are biregular mappings for $i = 1, \ldots, \mu(f)$. 

We \emph{claim that $q_j|_{V_i} : V_i \to \C$ is constant for each $i \neq j$ in $\{1, \ldots, \mu(f)\}$.}
Indeed, if this is not true for some $i$ and $j$ and $\psi: \C^{n-1} \to Z$ is the biregular mapping from the assumption, the non-constant polynomial $q_j \circ f|_{V_i}^{-1} \circ \psi : \C^{n-1} \to \C$ has a zero, and so $V_i \cap  V_j \neq \emptyset$.
This contradiction proves the claim.
	
So, since $g$ is nonsingular, it follows from the Nullstellensatz that $q_j = \beta_{i j} q_i + \alpha_{i j}$, for polynomials $\beta_{i j}$ and constants $\alpha_{i j}$.
It is simple to conclude that $\beta_{i j}$ are constants and so $g = P(q_1)$, with $P$ a polynomial of degree $\mu(f)$.
Since $g$ is nonsingular it follows that $\mu(f) = 1$.
Therefore, $f$ is injective and hence it is an automorphism from \cite{CR}.
\end{proof}

Pinchuk \cite{Pi} presented a nonsingular polynomial mapping $f:\R^2 \to \R^2$ that is not invertible, providing thus a counterexample to the {\it real Jacobian conjecture}. 
In this example, we have $S_f \cap Z_c = \emptyset$, for any line $Z_c := \{ (c,y) \mid y\in \R\}$ and $c < -1$, see for instance \cite{Cam}. 
Therefore, Theorem \ref{t:hyper} does not hold for nonsingular polynomial mapping $f: \R^n\to\R^n$. 

We now provide the 
\begin{proof}[Proof of Theorem \ref{t:Ra-topV}] 
Since $B (F_{\widehat{n}})$ is contained in a hypersurface $Z\subset \C^{n-1}$ (\cite{J03} or \cite[Corollaire 5.1]{Ve}), it follows from statement \ref{enSfBf} of Theorem \ref{t:proper} that $S_f \subset \pi_{\widehat{n}}^{-1}(Z)$. 
Let $y \notin  \pi_{\widehat{n}}^{-1}(Z)$. 
From \eqref{Jf} and \cite{CR}, it is enough to prove that $\# f^{-1}(y) = 1$. 

From the hypothesis, the fiber $F_{\widehat{n}}^{-1} (\pi_{\widehat{n}}(y))$ is simply connected. 
It thus follows by Proposition \ref{p:biregC} that $f_n$ is injective in this fiber. 
So $\# f^{-1}(y) = 1$, and we are done. 
\end{proof}

In the following remark we present an application of Theorem \ref{t:Ra-topV}: 
\begin{remark}\label{r:Hu}
An important result on the Jacobian conjecture given by Bass, Connel and Wright in \cite{BCW} is 
that the Jacobian conjecture in all dimensions follows if one proves that for all $n \geq 2$, nonsingular polynomial mappings of the form $f = I + H\colon\C^n\to\C^n$, where $I$ is the identity mapping and $H$ is a homogeneous polynomial of degree three, are injective. 
In \cite{hubbers}, Hubbers classified the nonsingular polynomial mappings $I+H$, for $n = 4$, up to linear conjugations, obtaining $8$ families of mappings. 
Then he  proved the bijectivity of each family by applying a criterion described by van den Essen in \cite{E1}, which is based on the calculation of Gr\"obner basis of an ideal defined from the components of $f$. 
Here we give a new and topological proof of the bijectivity of each mapping in Hubbers' classification, using Theorem \ref{t:Ra-topV}. 
Indeed, with the enumeration of \cite[Theorem 2.7]{hubbers} or \cite[Theorem 7.1.2]{E}, it is straightforward to check the simply connectedness of the fibers of:
\begin{itemize}\label{p:hubber}
\renewcommand{\labelitemi}{\ }
\item $F_{\widehat{4}} = (f_1,f_2,f_3)\colon\C^4\to\C^3$ for the families 1), 2), 7) and 8), 
\item $F_{\widehat{2}} = (f_1,f_3,f_4)\colon\C^4\to\C^3$ for the family 3), 
\item $F_{\widehat{3}} = (f_1,f_2,f_4)\colon\C^4\to\C^3$ for the families 4), 5) and 6). 
\end{itemize}
Thus each family is an automorphism by Theorem \ref{t:Ra-topV}. 
\end{remark}

Now, we can also do the 

\begin{proof}[Proof of Theorem \ref{t:main}] By applying statement \ref{enSf} of Theorem \ref{t:proper} for each $k$, it follows there is a set $B\subset \C$ such that $S_f = B\times \C^{n-1}$.
Since $S_f$ is either empty or it is a hypersurface, it then follows there exists $z\in\C$ such that the affine hyperplane $Z=\{z\}\times\C^{n-1}$ is disjoint of $S_f$. 
The result thus follows by Theorem \ref{t:hyper}.
\end{proof}

It is well known that analytic and geometric conditions can be used to estimate $B(F_k)$, see for instance \cite{DRT, J03, KOS, Ra}. Thus, we may use these conditions to ensure the topological hypothesis related to the $B(F_{\widehat{k}})$ in theorems \ref{t:Ra-topV}, \ref{t:main} and  \ref{t:proper}.  
 
\begin{remark}\label{r:MV-S}\rm
Splitting a complex mapping $f:\C^n\to\C^n$ into real and imaginary parts we obtain an associated real mapping $f^{\R}: \R^{2n}\to\R^{2n}$. In \cite[Corollary 2]{MV-S} it was proved that: \emph{the (complex) Jacobian conjecture is equivalent to the simply connectedness of all connected components of fibers of $f^\R_{i_1\ldots i_{2n-2}}$, for all combinations $(i_1<\ldots <i_{2n-2})$ of $\{1,\ldots,2n\}$.} Note that this requires to verify such topological condition for $2n^2-n$ mappings from $\R^{2n}\to\R^{2n-2}$ for mixing real and imaginary parts.
	Our Theorem \ref{t:main} improves this equivalence by proving that it is enough to check the same topological condition just for $n-1$ mappings instead of the $2n^2-n$ cases of \cite{MV-S}.
\end{remark}

%%%%%%%%%%%%%%%%%%%%%%%%%%%%%%%%%%%%%%%%%%%%%%%%%%%%%%%%

\section{On Rabier condition}
In this section we recall the definition of the set $\widetilde K_\infty(g)$, for holomorphic mappings $g: \C^n \to \C^m$. 
We also present the example of a polynomial automorphism $f$ in $\C^3$ such that $\widetilde K_\infty(F_{\widehat{k}}) \neq \emptyset$ for $k = 1, 2, 3$, as mentioned in the introduction section. 
We end this section discussing about our contributions related to already known results.  

\begin{definition}[\label{d:ra-ity}\cite{Ra}]
Let $g\colon\C^n\to\C^m$ be a polynomial mapping, with $n\geq m$.
We set
\begin{equation}\label{eq:ra-ity}
\begin{aligned}
\widetilde K_{\infty}(g) := & \{t\in\C^{m}\mid \exists \{ x_{j}\}_{j\in \mathbb{N}} \subset \C^{n}, \lim_{j\to\infty}\|x_j\|=\infty, \\ 
& \phantom{\{} \lim_{j\to\infty}g(x_j)= t \mathrm{\ and\ }\lim_{j\to\infty} \nu(\Jac(g) (x_j))=0\},
\end{aligned}
\end{equation}
where $\nu(A):=\inf_{\|\varphi\|=1}\|A^*(\varphi)\|$, for a linear mapping $A\colon\C^n\to\C^m$ and its adjoint $A^*\colon(\C^m)^*\to(\C^n)^*$. 
We say that $g$ satisfies the {\it Rabier condition} if $\widetilde K_{\infty}(g)=\emptyset$.
\end{definition}

For  $g:\C^n \to \C$, we have  $ \nu (\nabla g (x)) =   \|\nabla g (x)\|$ and if $g$ is nonsingular, Definition \ref{d:ra-ity} recovers the classical \emph{Palais-Smale condition}.

We observe that different functions instead of $\nu$ produces the same set $\widetilde K_{\infty}(g)$, see for instance \cite{J03, KOS}. 
Other conditions related to $\widetilde K_{\infty}(g)$ can be found for instance in \cite{DRT, KOS}.

The next example shows that for $f:\C^n \to\C^n$, $n\geq 3,$ a version of Theorem \ref{t:Ra}-(a) does not hold if we use the Rabier condition on the mappings $F_{\widehat{k}}:\C^n \to\C^{n-1}$. 

\begin{example}[See \cite{PZ}]\label{ex:noGPS} Let $f = (f_1, f_2, f_3):\C^3 \to \C^3$ be defined by
$$
f_1(x,y,z) = x + y h(x,y,z), \ \ \ \ f_2(x,y,z) = y, \ \ \ \ f_3(x,y,z) = h(x,y,z).
$$
where $h(x,y,z) = z - 3 x^5 y + 2 x^7 y^2$.
We have that $
\det \Jac (f) \equiv 1,$
and that $f$ is an automorphism whose inverse is
\[
f^{-1}(p,q,r) = \left(p-qr, q , r+3q(p-qr)^5 -2q^2(p-qr)^7\right).
\]
We also have  that $F_{\widehat{3}}=(f_1,f_2):\C^3\to\C^2$, $F_{\widehat{2}}=(f_1,f_3):\C^3\to\C^2$ and $F_{\widehat{1}}=(f_1,f_3):\C^3\to\C^2$ do not satisfy the Rabier condition, see Definition \ref{d:ra-ity}. 	In fact, to prove that $\widetilde K_{\infty}(F_{\widehat{3}})\neq \emptyset$, we may use the path $\lambda(t) = (t,1/t^2,0)$, as $t\to\infty$. 
For $F_{\widehat{2}}$ and $F_{\widehat{1}}$, we may use the paths $\gamma(t)=(1/t,t^2,1/t^3)$ and $\delta(t)=(t,1/t^2,t^3)$, respectively. 
\end{example}

\begin{remark}\label{r:gra-lw} 
The proof from \cite{LW} of Theorem \ref{t:LW} depends of the Abhyankar-Moh's result. 
On the other hand, the proof of Theorem \ref{t:Ra}-(b) presented in \cite{Ra} depends of a formula by Adjamagbo and van den Essen \cite[Corollary 1.4]{AV}. 
It is known that a nonsingular polynomial function $f:\C^2 \to \C$ satisfies the Rabier condition (i.e. $\widetilde K_{\infty}(f)=\emptyset$) if and only if $f$ is a locally trivial fibration, see \cite{Ha, Pa, ST2}. 
Thus, the proof of Theorem \ref{t:Ra-topV} gives also different proofs for theorems \ref{t:LW} and \ref{t:Ra}.  
\end{remark} 

%%%%%%%%%%%%%%%%%%%%%%%%%%%%%%%%%%%%%%%%%%%%%%%%%%%%%%%%

\section{The bidimensional case}\label{subsec:vf}

\begin{proof}[Proof of Proposition \ref{teoR2}]
Clearly \ref{en:b} implies \ref{en:a}. 

Assume \ref{en:a}. 
Without loss of generality we assume that $c = 0$ and write $f_1^{-1}(0) = V_1 \cup \cdots \cup V_d$ the decomposition of $f_1^{-1}(0)$ into its connected components. 
Each $V_j$ is an irreducible component of $f_1^{-1}(0)$, and so $V_j = q_j^{-1}(0)$, where $q_j$ is an irreducible polynomial. 
As in the proof of Proposition \ref{p:biregC}, it follows from the Riemann mapping theorem and \cite[Theorem $4$]{RW} that each connected component of $f_1^{-1}(c)$ is biregular to $\C$. 
Since $f_1$ is nonsingular, it follows that $f_1 = \gamma q_1 \cdots q_d$ with $\gamma \in \C$. 
As in the proof of Theorem \ref{t:hyper}, we conclude that $d = 1$, and so $f_1^{-1}(0)$ is biregular to $\C$. 
Now from \cite{AM}, it follows there exists an automorphism $h: \C^2 \to \C^2$ such that $g_1(x_1, x_2) = f_1 \circ h(x_1, x_2) = x_1$. 
Let $g_2(x_1, x_2) = x_2$ and define $f_2(x_1, x_2) = g_2\circ h^{-1}(x_1,x_2)$. 
Then $(f_1, f_2)$ is an automorphism. 
\end{proof}

As we said in the introduction section, the assumptions in  theorems \ref{t:Ra-topV} and \ref{t:main} on a nonsingular polynomial function $f_1: \C^2 \to \C$ are equivalent. 
An open question is to know if a nonsingular polynomial mapping $F_{\widehat{k}}: \C^n \to \C^{n-1}$ whose fibers have connected components simply connected have necessarily connected fibers.

\section*{Acknowledgments}
The first author was partially supported by the Fapesp Grant 2017/00136-0.
The second author was partially supported by the Fapemig-Brazil Grant APQ-00431-14 and CNPq-Brazil grants 401251/2016-0 and 304163/2017-1.
The third author was partially supported by the Fapemig-Brazil Grant APQ-00595-14 and the CNPq-Brazil Grant 446956/2014-7.

%%%%%%%%%%%%%%%%%%%%%%%%%%%%%%%%%%%%%%%%%%%%%%%%%%%%%%%%

\end{document}